\documentclass[11pt]{amsart}
\pdfoutput=1
\usepackage[shortlabels]{enumitem}

\usepackage{amsmath,amsfonts,amssymb,amsthm,mathrsfs}
\usepackage[colorlinks=true,linkcolor=red,citecolor=red]{hyperref}
\usepackage{bm}
\usepackage{mathabx}
\usepackage{color}
\usepackage{esint}
\usepackage{graphicx}
\usepackage{graphics}
\usepackage{geometry}
\usepackage[all]{xy}
\usepackage{tikz}
\usepackage{tikz-cd}
\usepackage{verbatim}
\usepackage{pgfplots}

\usetikzlibrary{decorations.pathreplacing}
\usetikzlibrary{intersections, pgfplots.fillbetween}

 \newtheorem{theorem}{Theorem}[section]
 \newtheorem{lemma}[theorem]{Lemma}
 \newtheorem{corollary}[theorem]{Corollary}
 \newtheorem{proposition}[theorem]{Proposition}
 \newtheorem{conjecture}[theorem]{Conjecture}
 \newtheorem{remark}[theorem]{Remark}
 
 \theoremstyle{definition}
 
 \newtheorem{question}[theorem]{Question}
 
 \newtheorem{problem}[theorem]{Problem}

\numberwithin{equation}{section}

\newcommand{\p}{\partial}

\newcommand{\C}{\mathbb{C}}

\newcommand{\R}{\mathbb{R}}

\begin{document}
\title{On conical asymptotically flat manifolds}

\author{Mingyang Li}
\address{
Department of Mathematics, University of California, Berkeley, CA 94720}
\email{mingyang\textunderscore li@berkeley.edu}
 \author{Song Sun}
\address{Institute for Advanced Study in Mathematics, Zhejiang University, Hangzhou 310058, China}
\address{
Department of Mathematics, University of California, Berkeley, CA 94720} 
\email{songsun@zju.edu.cn}

\dedicatory{Dedicated to Professor Xiaochun Rong for his 70th birthday}


\begin{abstract}
	We prove a conjecture of Petrunin and Tuschmann on the non-existence of asymptotically flat 4-manifolds asymptotic to the half plane. We also survey recent progress and questions concerning gravitational instantons, which serve as our motivation for studying this question. 
\end{abstract}

\maketitle

\section{Introduction}
In this paper we study conical asymptotically flat ($\mathcal{AF}$) manifolds. By definition, a complete non-compact $m$ dimensional Riemannian manifold $(M^m, g, p)$ is called 
 \begin{itemize}
 \item  $\mathcal{AF}$  if the Riemannian curvature decays at the rate $|Rm_g|=o(r^{-2})$ as $r\rightarrow \infty$, where $r$ denotes the distance to $p$; in other words, the \emph{asymptotic curvature} $$A(M, g)\equiv \limsup_{r\rightarrow\infty} r^2|Rm_{g}|$$ is zero.
 \item  \emph{conical} if  it is asymptotic to a unique metric cone $\mathcal C_\infty$ at infinity, i.e., $(M, \lambda^{-2}g, p)$ converges in the pointed Gromov-Hausdorff topology to $\mathcal C_\infty$ as $\lambda\rightarrow\infty$.   Here the  \emph{asymptotic cone} $\mathcal C_\infty$  may have a lower dimension. 
 \end{itemize}
We need to make a few remarks about the terminologies here. First  the words ``asymptotically flat" may refer to properties of varying generality in the literature. The notion of $\mathcal{AF}$ we use here is the same as the notion of asymptotically flat used in \cite{PT}. A related notion is what we shall call \emph{strongly $\mathcal{AF}$}, which means that $|Rm_g|$ is bounded by a positive decreasing function $f(r)$ with $\int_0^\infty rf(r)dr<\infty$. An even stronger condition is that of \emph{faster than quadratic curvature decay}, which means that $|Rm_g|\leq r^{-2-\epsilon}$ for some $\epsilon>0$ as $r\rightarrow\infty$.  We also try to distinguish the notation $\mathcal{AF}$ from AF, which usually refers to being asymptotic to a specific flat  model end (see Section \ref{sec:discussion}). 

There has been extensive work in Riemannian geometry studying topological and geometrical properties of $\mathcal{AF}$ manifolds. Recall that flat ends of Riemannian manifolds were completely classified by Eischenburg-Schroeder \cite{ES}. Ends of $\mathcal{AF}$ manifolds are much more flexible; for example, any 2 dimensional non-compact surface admits an $\mathcal{AF}$ metric \cite{Abresch1}. It is known \cite{Kasue, MNO} that strongly $\mathcal{AF}$ manifolds are automatically conical. However there are interesting examples of conical $\mathcal{AF}$ manifolds which are not strongly $\mathcal{AF}$. Examples include the family of ALG$^*$ hyperk\"ahler metrics \cite{Hein}.

Our interest in $\mathcal{AF}$ manifolds is motivated by the fact that they provide potential model ends for complete Ricci-flat metrics.  In 4 dimensions a complete non-compact Ricci-flat manifold with $\int |Rm_g|^2<\infty$ is called a \emph{gravitational instanton}. The readers should be warned that there are varying definitions of gravitational instantons in the literature. Some require the extra assumption of being hyperk\"ahler, but we do not make this hypothesis in the current paper.  All known examples of gravitational instantons are conical and they are all $\mathcal{AF}$ except the family of ALH$^*$ hyperk\"ahler metrics (see \cite{Hein, SZ}). The latter are known to have  \emph{nilptoent} asymptotic geometries.

For general conical $\mathcal{AF}$ manifolds, in 2001 Petrunin and Tuschmann \cite{PT} proved a structural result regarding the end structure. In particular, they proved that there are only finitely many ends and a simply connected end must have asymptotic cone the flat $\R^m$ if $m\neq 4$ and the flat $\R^4, \R^3$ or the half plane $\mathbb H\equiv \R\times [0, \infty)$ when $m=4$. It follows that when $m\neq 4$ the metric has Euclidean volume growth and there is no \emph{collapsing} phenomena along the convergence to the asymptotic cone. But when $m=4$ collapsing may indeed occur so the situation is more interesting.  Clearly $\R^4$ can be realized. Also $\R^3$ can be realized as the asymptotic cone of the Taub-NUT gravitational instanton (see also \cite{Unnebrink}), whose end is diffeomorphic to $S^3\times [0, \infty)$.  Along the convergence to the asymptotic cone $\mathcal C_\infty$, the $S^3$ collapses along the Hopf fibration to $S^2$.   For the half plane $\mathbb H$, we know that it can be realized as the asymptotic cone of the flat space $\R^4/\mathbb Z$ (and of some members of the Kerr family of gravitational instantons), where the generator of the $\mathbb Z$-action is given by a translation in the first $\R^2$ factor and an irrational rotation in the second $\R^2$ factor. The end is diffeomorphic to $(S^1\times S^2)\times [0, \infty)$; in particular it is  \emph{not} simply-connected.  Petrunin-Tuschmann further made the following conjecture

\begin{conjecture}[Petrunin-Tuschmann \cite{PT}]\label{conj:PT}
There is no conical $\mathcal{AF}$  4-manifold with a simply-connected end whose  asymptotic cone is the half plane. 
\end{conjecture}

Notice that there is no curvature equation imposed. It is easy to see that if a conical $\mathcal{AF}$ 4-manifold has an end with asymptotic cone $\mathbb H$, then the end must be diffeomorphic to $N\times [0, \infty)$ where $N$ is diffeomorphic to either $S^1\times S^2$ or $S^3/\Gamma$. So the confirmation of Conjecture \ref{conj:PT} would give a topological classification of such ends. On the other end, if Conjecture \ref{conj:PT} is false then potentially there could be new interesting gravitational instantons with this asymptotics. 

 The results in \cite{PT} were proved by investigating the geometry of the rescaled annuli $(M, \lambda^{-2}g, p)$ as $\lambda\rightarrow\infty$ when they collapse to lower dimensions. Foundational results on collapsing  were obtained by Cheeger, Fukaya, Gromov, and others (see for example \cite{CFG, Rong}). Specifically, \cite{PT} makes use of the \emph{continuously} collapsing techniques developed by Petrunin-Rong-Tuschmann \cite{PRT}.  In particular, the argument is of \emph{local} nature, i.e., it focuses on the local collapsing geometry of the family of almost flat annuli of fixed size (with respect to the rescaled metrics), and it does not use the fact that the annuli come from the end of a \emph{fixed} $\mathcal{AF}$ manifold. In fact as mentioned in \cite{PT} (see also Section 3), one may construct a sequence of Riemannian metrics $g_i$ on $\mathfrak A=S^3\times [0, 1]$ with $\sup_{A}|Rm_{g_i}|\rightarrow 0$ which collapse to the annulus  $\mathbb A\equiv\{2^{-1}\leq r\leq 2\}$ in $\mathbb H$, where $r$ is the distance function to the origin in $\mathbb H$. Therefore Conjecture \ref{conj:PT} is of \emph{global} nature. We remark that  even under the stronger assumption of faster than quadratic curvature decay Conjecture \ref{conj:PT} was unknown.

In this paper we confirm the conjecture of Petrunin-Tuschmann.

\begin{theorem}\label{thm}
Conjecture \ref{conj:PT} holds. 
\end{theorem}

We will see that the standard $\R^4$ admits a complete conical Riemannian metric $g_j$ with asymptotic curvature $A(\R^4, g_j)\leq j^{-1}$ and asymptotic cone $\mathbb H$, but by Theorem \ref{thm}  there is no complete conical $\mathcal{AF}$ metric on $\R^4$ with asymptotic cone $\mathbb H$. This also gives a negative answer to a version of the gap question in \cite{PT} (Question 2). 

\

Now we briefly sketch the idea behind the proof of Theorem \ref{conj:PT}. Suppose $(M, g)$ is a conical $\mathcal{AF}$ 4-manifold with a simply-connected end whose asymptotic cone is $\mathbb H$, we will derive a contradiction.  General theory allows us to reduce to the case when the end admits a $T^2$ action and the metric $g$ is $T^2$ invariant. Such a metric $g$ then yields a family of flat metrics $G$ on the  2-torus parametrized by $(\rho, z)\in \mathbb H=\mathbb R_{\geq0}\times \R$ for $\rho^2+z^2\geq R^2$. For fixed $z$ as $\rho\rightarrow  0$ the metric $G$ degenerates as in a specific model situation that we can understand. Fix  an interval $I\subset(0, \pi)$ and consider the sector region $z=\rho \cot\theta$ for $\theta\in I$.  As $\rho\rightarrow\infty$, the 2-torus endowed with the rescaled metric $\rho^{-1}G$  converges to a point in the Gromov-Hausdorff sense. We introduce two quantities $\sigma$ and $\tau$ to each metric $G$. Roughly speaking, $\sigma$ measures the area of the 2-torus under the metric $\rho^{-1}G$, and $\tau$ measures the deviation of the metric $\rho^{-1}G$ from a model flat metric. The fact that the asymptotic cone is $\mathbb H$ gives control on the asymptotics of quantities involving $\tau$ and $\sigma$ as $\rho\rightarrow\infty$. Then an elementary calculus argument yields a contradiction.

\

\textbf{Acknowledgements:} Both authors were partially supported by the Simons Collaboration Grant in Special Holonomy. The first author is grateful to the IASM at Zhejiang University for hospitality during his visit in Spring 2024, and would like to thank Prof. Guofang Wei for sharing Abresch's paper \cite{Abresch}. We thank John Lott for helpful comments that improved the exposition of the paper. 

\section{Proof of the main result}
\label{sec:proof}
Suppose $(M, g, p)$ is a conical $\mathcal{AF}$ manifold with a simply-connected end and with asymptotic cone $\mathcal C_\infty=\mathbb H$.  We will derive a contradiction. 

First we introduce some notations. Let $(\rho, z)$ be the standard coordinates on $\mathbb H\equiv \R_{\geq 0}\times \R$.  Denote by $g_0=d\rho^2+dz^2$ the standard metric on $\mathbb H$ and by $r_0$ the distance function to the origin. 
For $j>0$ write $A_j\equiv \overline B_g(p, 2^{j+1})\setminus B_g(p, 2^{j-1})$ and denote by $\widetilde A_j$ the manifold $A_j$ endowed with the rescaled metric $g_j\equiv 2^{-2j}g$.  By Theorem A of \cite{PT} we know $A_j$ is simply connected for $j$ large. 

As initial steps we make a few reductions using general theory. We first prove a smoothing result.

\begin{lemma}
	\label{p:smoothing}
	For any given $k_0>0$ there is a conical $\mathcal{AF}$ Riemannian metric $g'$ on $M$ with asymptotic cone $\mathbb H$ such that for each $k\leq k_0$, $|\nabla_{g'}^k Rm_{g'}|_{g'}=o(r^{-k-2})$ as $r\rightarrow 0$.  
\end{lemma}

\begin{remark}
Notice that this higher regularity is automatic if we assume $g$ satisfies elliptic curvature equations, for example, if we assume $g$ is Ricci-flat, or if we assume $g$ is scalar-flat and anti-self-dual (which includes the case of being scalar-flat K\"ahler).\end{remark}

\begin{proof}
	Fix $\epsilon>0$ small to be determined below. Denote by $\eta_j$ the maximum of $|Rm_{g_j}|$ on $\widetilde A_{j}\cup \widetilde A_{j-1}\cup \widetilde A_{j+1}$. By assumption we have $\eta_j\rightarrow 0$ as $j\rightarrow\infty$. 
	By the local smoothing result of Abresch \cite{Abresch} (more precisely, by iterating the proof of Theorem 1.1 in \cite{PWY} for finitely many steps) to each rescaled annuli $\widetilde A_j$, one can find a new Riemannian metric $g'_j$ on $\widetilde A_j$ with $(1-\epsilon)g_j\leq g'_j\leq (1+\epsilon)g_j$, $|Rm(g'_j)|\leq 2 \eta_j$ and $|\nabla^k_{g'_j}Rm(g'_j)|\leq C_k$ for each $k\leq k_0$. Scale $g'_j$ back to $A_j$ and glue the adjacent $g'_j$ by suitable cut-off functions (as in \cite{CT}) we obtain a metric $g'$ on $M$ such that $(1-\epsilon)g\leq g'\leq (1+\epsilon)g$ and $|\nabla^k_{g'}Rm_{g'}|_{g'}=o(r^{-2-k})$ as $r\to\infty$ for all $k\leq k_0$.

It remains to show that $g'$ is asymptotic to $\mathbb H$ if $\epsilon$ is chosen to be small. Denote by $\widetilde A_j'$ the manifold $A_j$ endowed with the rescaled metric $g'_j\equiv 2^{-2j}g'$. Due to the uniform equivalence between $g$ and $g'$,  by passing to a subsequence we may assume $(\widetilde A_j', g'_j, p)$ converges to a pointed Gromov-Hausdorff limit space $(\mathcal C, O)$, which is isometric to $\mathbb H$ endowed with a possibly different metric $d_\infty$. In particular $\mathcal C$ is homeomorphic to $\mathbb H$ but a priori we do not know if $\mathcal C$ is a metric cone. 

By the assumption on $(M, g)$  locally for any $q_j\in \widetilde A_j$ converging to $q_\infty\in \mathbb H$ we may find $r$ small such that the universal cover $(\widehat B_j, \widehat g_j)$ of $B_j\equiv B_{g_j}(q_j,r)\subset \widetilde A_j$ is uniformly non-collapsed and there is an equivariant $C^{1, \alpha}$ Gromov-Hausdorff convergence of $(\widehat B_j, \widehat g_j, \Lambda_j, \widehat{q}_j)$ to a non-collapsing flat limit space $(\widehat B_\infty, \widehat g_\infty, \Lambda_\infty, \widehat{q}_\infty)$, where $\Lambda_j=\pi_1(B_j)$ such that  $(\widehat B_\infty, \widehat g_\infty)/\Lambda_\infty$ is identified with the ball $B_\infty\equiv B_{g_0}(q_\infty, r)$ in $\mathbb H$. We may isometrically immerse $\widehat B_\infty$ into $\R^4$. Then the group $\Lambda_\infty$ descends to a nilpotent (hence abelian) subgroup $\Lambda$ of the group of Euclidean motions of $\R^4$. Since $A_j$ is simply connected we know by \cite{Rong} that $\Lambda_\infty$ is connected.  By the Sublemma in \cite{PT} there is an orthogonal decomposition $\mathbb{R}^4=V\oplus W$ such that $\Lambda_\infty\subset Rot(W)\oplus Trans(V)$. Using the Key Lemma in \cite{PT} (in the statement of the Key Lemma in \cite{PT} we know $l<m$, which is ultimately due to the injectivity radius estimate in \cite{PRT}) we see that $\dim W=\dim V=2$, and $\Lambda$ is generated by a rotation in $W$ and a translation in $V$, so it is abstractly isomorphic to $S^1\times \R$. It is also easy to identify the $\Lambda$ for different choices of $q_\infty$.

Now choose $q_\infty\in \p\mathbb H$. Then we can identify the corresponding $\Lambda_\infty$ with $\Lambda$.  Denote by $\widehat g_j'$ the induced metric on $\widehat B_j$ from $g'_j$. Then we may assume the metric $\widehat g_j'$ also converges  to a flat metric $\widehat g_\infty'$ on $\widehat B_\infty$ which is also $\Lambda$ invariant and the quotient $(\widehat B_\infty, \widehat g_\infty')/\Lambda_\infty$ is identified with the corresponding ball $B_\infty'$ in $(\mathbb H, d_\infty)$.  Similarly by the Sublemma in \cite{PT} we may assume that with respect to the metric $\widehat g_\infty'$ we also have $\Lambda\subset Rot(W')\oplus Trans(V')$. Since $\Lambda$ contains a factor isomorphic to $S^1$ it follows that $\Lambda_\infty$ is also generated by a rotation in $W'$ and another element $v$ in $Rot(W')\oplus Trans(V')$. Without loss of generality we may assume  $v$ is a  translation in $V'$. It then follows that $B_\infty'$ must be flat hence $\mathcal C$ is flat  and $\p\mathcal C$ is totally geodesic (hence a straight line) away from $O$.  So $\mathcal C$ is isometric to a wedge. To see $\mathcal C$ is indeed isometric to $\mathbb H$, we choose a path $\gamma(t)$ in the interior of $\mathbb H$ connecting $q_\infty$ and $-q_\infty$. One can cover $\gamma(t)$ by finitely many small balls $\{D_\alpha\}, \alpha=1, \cdots, s$, where $q_\infty\in D_1$, $-q_\infty\in D_s$ and $D_{\alpha}\cap D_{\alpha+1}\neq\emptyset$. Now each $D_\alpha$ arises as the quotient of the form $\widehat D_{\alpha, \infty}/\Lambda_{\infty, \alpha}$, where $\widehat D_{\alpha, \infty}$ is the flat limit of local universal covers of corresponding balls in $\widetilde A_j'$ and $\Lambda_{\infty, \alpha}$ can be identified as a subgroup of $Rot(W_\alpha)\oplus Trans(V_\alpha)$ for an orthogonal decomposition $\mathbb R^4=W_\alpha\oplus V_\alpha$, under a local isometric immersion of $D_{\alpha, \infty}$ into $\R^4$. Now by considering the intersection $D_\alpha\cap D_{\alpha+1}$ and starting from $\alpha=1$ one can naturally all $W_\alpha$ and $V_\alpha$ for different $\alpha$, and it follows that  $\Lambda_{\infty, \alpha}$ is generated by a common rotation and translation. From this it is easy to see that $\mathcal C$ is isometric to $\mathbb H$.

	\end{proof}

 We will assume from now on that  $g'=g$. Next we apply the Cheeger-Fukaya-Gromov theory to reduce the case when the end of $g$ is $T^2$ invariant. 
\begin{lemma}\label{prop:CFG}
		There is an $R>0$, an effective $G=T^2$ action on $M^\circ\equiv M\setminus B(p, R)$, a $G$-invariant Riemannian metric $\widecheck g$ on $M^\circ$ and a surjective continuous map $F: M^{\circ}\simeq \mathbb H^\circ\equiv \mathbb H\setminus K$ for some compact $K$, such that the following hold
	\begin{itemize}
		\item $\widecheck g$ is close to $g$ in the sense that for each $k\leq k_0$, $|\nabla_g^k(\widecheck g-g)|_g=o(r^{-k})$ as $r\rightarrow\infty$; in particular, $\widecheck g$ is also conical $\mathcal{AF}$ with asymptotic cone $\mathbb H$. 
		\item For each $x\in \mathbb H^\circ$, $F^{-1}(x)$ consists of exactly one $G$ orbit and $M^\circ/G$ is a smooth manifold. Moreover, $F$ induces a diffeomorphism between $M^\circ/G$ and $\mathbb H^\circ$, and $\widecheck g$ descends to a smooth Riemannian metric $g_b$ on the quotient $M^\circ/G$  (in the sense of manifold with boundary). 
		\item Denote by $P\equiv P_1\sqcup P_2$ the two boundary components of $\mathbb H^\circ$, then the $G$ action is free over $F^{-1}(\mathbb H^\circ\setminus P)$ and points in $F^{-1}(P_\alpha)$ have stabilizers isomorphic to a 1-dimensional subgroup $G_\alpha\subset T^2$ with $G_1\cap G_2=\{e\}$. 
		\item $F$ is an asymptotic isometry in the sense that under the above identification between $M^\circ/G$ and $\mathbb H^\circ$, we have  for each $k\leq k_0$, $|\nabla_{g_0}^k(g_b-g_0)|_{g_0}=o(r_0^{-k})$ as $r_0\rightarrow\infty$. 
	\end{itemize}
\end{lemma}
\begin{proof}
This is essentially a standard application of the Cheeger-Fuakay-Gromov theory, as in \cite{SZ}. We only give a sketch of the argument here.  One first works with the collapsing annuli $\widetilde A_j$ and apply \cite{CFG} to obtain an $\mathcal N$-structure in $\widetilde A_j$ (slightly shrinking if necessary) for $j$ large. The $\mathcal N$-structure is \emph{pure} since we have a diameter bound. By hypothesis  each $\widetilde A_j$ is simply-connected so by Rong \cite{Rong} the pure $\mathcal N$-structure is given by an action of $T^2$. Points with codimension 1-stabilizers give rise to boundary in the quotient. Again the simply-connectedness implies that the two boundary components correspond to two different stabilizer groups $G_1$ and $G_2$ with $G_1\cap G_2=\{e\}$. Then we can glue together the structures on different $\widetilde A_j$ to obtain a $T^2$ action on $M^\circ$ for $R$ large. 
The construction and properties of the invariant metric $\widecheck g$ follow from the averaging construction in \cite{CFG}. 
\end{proof}

Now we fix $k_0=10$. Without loss of generality we will assume $\widecheck g=g$ 
in the following. 
The above discussion only makes use of the topology of the end and of the asymptotic cone. Now we investigate the \emph{asymptotical flat} geometry in more detail. Introduce the following notations. 
\begin{itemize}
	\item We write the decomposition $G=G_1\times G_2$  and denote by $\p_{\alpha}$ the infinitesimal generator of $G_\alpha$, normalized to have period $2\pi$.
	\item Fix $\epsilon>0$ small, we may  cover $\mathbb H$ by 2 open subsets $\mathbb U_1=\{0\leq\theta<\frac\pi 2+\epsilon\}$ and $\mathbb U_2=\{\frac\pi 2-\epsilon<\theta\leq \pi\}$. Then we get an open cover of $M^\circ$ by $U_\alpha\equiv F^{-1}(\mathbb U_\alpha), \alpha=1, 2$. 
	\item Denote $U_{\alpha, j}\equiv U_\alpha\cap \widetilde A_j$, which is endowed with the rescaled metric $g_j$, so $\widetilde A_j=U_{1, j}\cup U_{2, j}$. Note that $U_{\alpha,j}$ is converging to $\mathbb{U}_{\alpha}\cap\mathbb{A}$.
	\item Denote $H_\alpha=\pi_1(U_{\alpha, j})$. Then $H_\alpha$ is isomorphic to $\mathbb Z$ and is generated by the orbit of the $G_{3-\alpha}$ action. Notice that since the $G_\alpha$ action on $U_{\alpha, j}$ has fixed points the $G_\alpha$ orbits are homotopically trivial in $U_{\alpha, j}$.
	\item Denote $\mathbb V\equiv\mathbb U_1\cap \mathbb U_2$ and $V\equiv F^{-1}(\mathbb V)$. Denote the overlap by $V_j\equiv U_{1, j}\cap U_{2, j}$, then $\pi_1(V_j)=H_{1}\times H_2$. Note that $V_{j}$ is converging to $\mathbb{V}\cap\mathbb{A}$.
\end{itemize}
See Figure \ref{Figure:cover the half plane} and Figure \ref{Figure:Singular fibration} for pictures of the singular fibration $F$.

\begin{figure}[ht]
	\begin{tikzpicture}[scale=0.6]
	
		\draw(0,-5)--(0,-1) arc (270:380:1) --(6,2.4);
		\draw(0,5)--(0,1) arc (-90:-200:-1) --(6,-2.4);
		\node at (5, 3) {$\mathbb{U}_2$};
		\node at (5, -3) {$\mathbb{U}_1$};

		\node at (-0.5, -4) {$P_1$};
		\node at (-0.5, 4) {$P_2$};

		\node at (5,0) {$\mathbb{V}$};

		\filldraw[color=red!50,fill=red!50,opacity=0.5](0,-5)--(0,-1) arc (270:380:1) --(6,2.4)--(6,-5)--(0,-5);
		\filldraw[color=blue!50,fill=blue!50,opacity=0.5](0,5)--(0,1) arc (-90:-200:-1) --(6,-2.4)--(6,5)--(0,5);

		\filldraw[color=purple!50,fill=purple!50,opacity=0.5](0,4)--(0,2) arc (-90:-270:-2) --(0,-4) arc (90:270:-4)--(0,4);

		\node at (3,0) {$\mathbb{A}$};

	\end{tikzpicture}
	\caption{Cover $\mathbb{H}^\circ$ by $\mathbb{U}_1$ and $\mathbb{U}_2$}
	 \label{Figure:cover the half plane}
\end{figure}
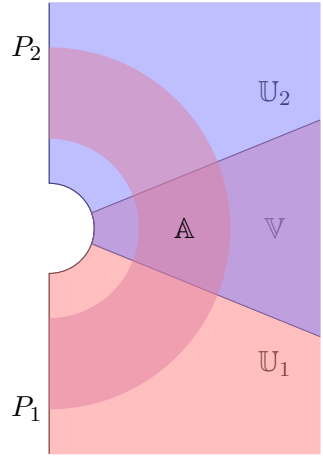

\begin{figure}[ht]
	\begin{tikzpicture}[scale=0.6]
	
		\draw(8,0)--(1,0) arc 
		[
			start angle=0,
			end angle=110,
			x radius=2cm,
			y radius=1cm
		]  --(-4.8,4);

		\filldraw[color=red!50,fill=red!50,opacity=0.5](8,0)--(1,0) arc 
		[
			start angle=0,
			end angle=110,
			x radius=2cm,
			y radius=1cm
		]  --(-4.8,4);

		\filldraw[color=blue!50,fill=blue!50,opacity=0.5](-1,4)--(-1,1) arc 
		[
			start angle=90,
			end angle=110,
			x radius=2cm,
			y radius=1cm
		]  --(-4.8,4);

		\node at (2, 1) {$\mathbb{U}_i$};
		\node at (4, -0.5) {$P_i$};
		\node at (-2,2) {$\mathbb{V}$};

		\draw[dashed, ->] (-3.5,5) -- (-3.5,3);
		\node at (-3.5,2.7) {$\bullet$};
		\draw (-2.5,7.2) arc
    	[
        	start angle=0,
        	end angle=360,
        	x radius=1cm,
        	y radius=2cm
    	];
		\draw (-3.5, 7.5) to [out=330, in=30] (-3.5, 6.5);
		\draw (-3.3,8) arc
    	[
        	start angle=110,
        	end angle=250,
        	x radius=0.5cm,
        	y radius=1cm
    	];
		\draw[densely dotted] (-3.6,7) arc
    	[
        	start angle=0,
        	end angle=360,
        	x radius=0.43cm,
        	y radius=0.2cm
    	];

		\draw[dashed, ->] (-0,4) -- (-0,2);
		\node at (-0,1.7) {$\bullet$};
		\draw (1,6.2) arc
    	[
        	start angle=0,
        	end angle=360,
        	x radius=1cm,
        	y radius=2cm
    	];
		\draw (-0.05, 7) to [out=310, in=50] (-0.05, 5.2);
		\draw (0.2,7.4) arc
    	[
        	start angle=120,
        	end angle=240,
        	x radius=1cm,
        	y radius=1.5cm
    	];
		\draw[densely dotted] (-0.3,6) arc
    	[
        	start angle=0,
        	end angle=360,
        	x radius=0.35cm,
        	y radius=0.17cm
    	];

		\draw[dashed, ->] (3.5,2.8) -- (3.5,0.8);
		\node at (3.5,0.5) {$\bullet$};
		\draw (4.5,5) arc
    	[
        	start angle=0,
        	end angle=360,
        	x radius=1cm,
        	y radius=2cm
    	];
		\draw (3.45, 6.4) to [out=310, in=50] (3.45, 3.6);
		\draw (3.6,6.6) arc
    	[
        	start angle=130,
        	end angle=230,
        	x radius=1.6cm,
        	y radius=2.1cm
    	];
		\draw[densely dotted] (3.05,4.7) arc
    	[
        	start angle=0,
        	end angle=360,
        	x radius=0.26cm,
        	y radius=0.13cm
    	];
		\draw[<-] (3.5,5.6) arc
    	[
        	start angle=300,
        	end angle=130,
        	x radius=0.9cm,
        	y radius=0.2cm
    	];
		\draw[<-] (4,2.8) arc
    	[
        	start angle=300,
        	end angle=230,
        	x radius=1cm,
        	y radius=2cm
    	];
		\node at (4.6,2.1) {$G_{3-i}$};
		\node at (1.8,5.2) {$G_i$};

		\draw[dashed, ->] (7,2.3) -- (7,0.3);
		\node at (7,0) {$\bullet$};
		\draw (7.6,4.2) arc
    	[
        	start angle=0,
        	end angle=360,
        	x radius=0.6cm,
        	y radius=1.6cm
    	];

		\draw[dashed, ->] (4.7,9) -- (-2.5,9);
		\draw[dashed, ->] (4.7,8.7) -- (1,8);
		\draw[dashed, ->] (5,8.6) -- (4,7);
		\node at (7.6,9) {smooth fibers $T^2$};

		\draw[dashed, ->] (9,5) -- (8,5);
		\node at (12,5) {singular fibers $S^1$};
		\node at (12,4) {isotropy group $G_i$};

	\end{tikzpicture}
	\caption{Singular fibration over $\mathbb{U}_i$}
	 \label{Figure:Singular fibration}
\end{figure}
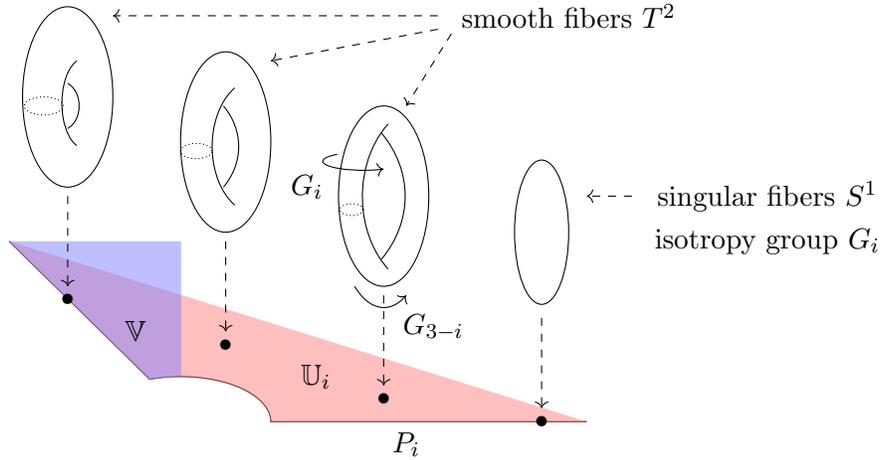

Now consider the convergence of $\widetilde A_j$ to $\mathbb A=\{1/4\leq \rho^2+z^2\leq 4\}\subset \mathbb H$. Denote the universal cover of $U_{\alpha, j}$ by $\widehat U_{\alpha, j}$. Then we see that for each $\alpha$, $\widehat U_{\alpha, j}$ (with the induced metric from $g_j$) is volume non-collapsing. Furthermore, $\p_1$ and $\p_2$ both lift  to Killing fields on $\widehat U_{\alpha, j}$, which we denote by $\widehat\p_{1; \alpha, j}$ and $\widehat\p_{2; \alpha, j}$ respectively. Notice that on $\widehat U_{\alpha, j}$ the vector field $\widehat\p_{\alpha; \alpha, j}$ still generates an $S^1$ action which descends to the action of $G_\alpha$ on $U_{\alpha, j}$.  Passing to subsequences we may assume that $(\widehat U_{\alpha, j}, H_{\alpha})$ converges equivariantly in the $C^{k_0}$ Cheeger-Gromov topology  to a 4-dimensional flat manifold $(\widehat U_{\alpha,\infty}, H_{\alpha,\infty})$ so that $\mathbb U_\alpha\cap\mathbb{A}$ is isometric to $\widehat U_{\alpha,\infty}/ H_{\alpha,\infty}$. 

Since we know the collapsing to $\mathbb A$ is along the $G$ orbits, we see that $H_{\alpha,\infty}$ is generated by two commuting Killing fields both of which are limits of linear combinations of $\widehat\p_{1; \alpha, j}$ and $\widehat\p_{2; \alpha, j}$.
Now since the $G_\alpha$ action  lifts to $\widehat U_{\alpha, j}$  it also acts on the limit $\widehat U_{\alpha,\infty}$, hence $\widehat \p_{\alpha; \alpha, j}$ naturally converges to an element $\widehat\p_{\alpha,\infty}$ in the Lie algebra $\mathfrak h$ of ${H}_{\alpha,\infty}$ with period $2\pi$.  Locally it is given by a rotation field in a 2-dimensional plane $W\subset \R^4$ after immersion into $\mathbb{R}^4$.
Given any other Killing field $\xi\in \mathfrak h$, since it commutes with $\widehat\p_{\alpha,\infty}$, locally it must be given by $\lambda \widehat\p_{\alpha,\infty}+\xi'$, where $\xi'$ is a Euclidean motion in the plane orthogonal to $W$.
We claim that $\xi'$ can not be a rotation. One way to see this is to apply the Key Lemma in \cite{PT} (again, in the statement of the Key Lemma in \cite{PT} we know $l<m$). 
Therefore we may identify a unique element (up to $\pm1$) $\widehat\xi_{\alpha}\in \mathfrak h$ with $\|\widehat\xi_{\alpha}\|=1$, which is locally a translation vector field, such that $\widehat{\partial}_{\alpha,\infty}$ and $\widehat{\xi}_\alpha$ generate $\mathfrak{h}$.


For  $x\in \widehat U_{\alpha,\infty}$ where $\widehat\p_{\alpha,\infty}$ does not vanish, we denote by $\kappa(x)$ the geodesic curvature of the orbit of the flow of $\widehat\p_{\alpha,\infty}$ at $x$. 

\begin{lemma}
For all such $x$, we have
\begin{equation}\label{e:lengthcurvature}
	\|\widehat\p_{\alpha,\infty}(x)\|\cdot |\kappa(x)|=1.
\end{equation}
\end{lemma}
\begin{proof}
	Let $x_0$ be a zero point of $\widehat\p_{\alpha,\infty}$. Then we can isometrically embed a neighborhood of $x_0$ into $\R^4$ such that $\widehat\p_{\alpha,\infty}$ is a standard rotation vector field with period $2\pi$. It is clear that the equality holds for $x$ close to $x_0$. Using the Gauss-Bonnet theorem it is easy to see the same also holds for all $x$. 
\end{proof}

Let $\widecheck V_j$ be the universal cover of $V_j$. 
Passing to subsequences again we may assume $(\widecheck V_j, H_1\times H_2)$ converges equivariantly  in the $C^{k_0}$ Cheeger-Gromov topology to a 4-dimensional flat manifold $(\widecheck V_\infty, \widecheck H)$, where $\widecheck H$ is abstractly isomorphic to $\mathbb R^2$. Denote by $\widehat V_{\alpha, j}$ the intermediate covering of $V_j$ associated to the subgroup $H_\alpha\subset H_1\times H_2$. Then $\widehat V_{\alpha, j}$ can be naturally identified as an open subset of $\widehat U_{\alpha, j}$. We may assume $(\widehat V_{\alpha, j}, H_\alpha)$ converges to $(\widehat V_{\alpha,\infty}, H_{\alpha,\infty})\subset (\widehat U_{\alpha,\infty}, H_{\alpha,\infty})$ and there is a $\mathbb Z$-covering map from $\widecheck V_\infty$ to $\widehat V_{\alpha,\infty}$. It follows that $\widecheck H$ is also  generated by a rotation vector field  and a translation vector field. See Figure \ref{Figure:Intermediate covers} for a topological picture of the intermediate covers and the universal cover for $V_j$.

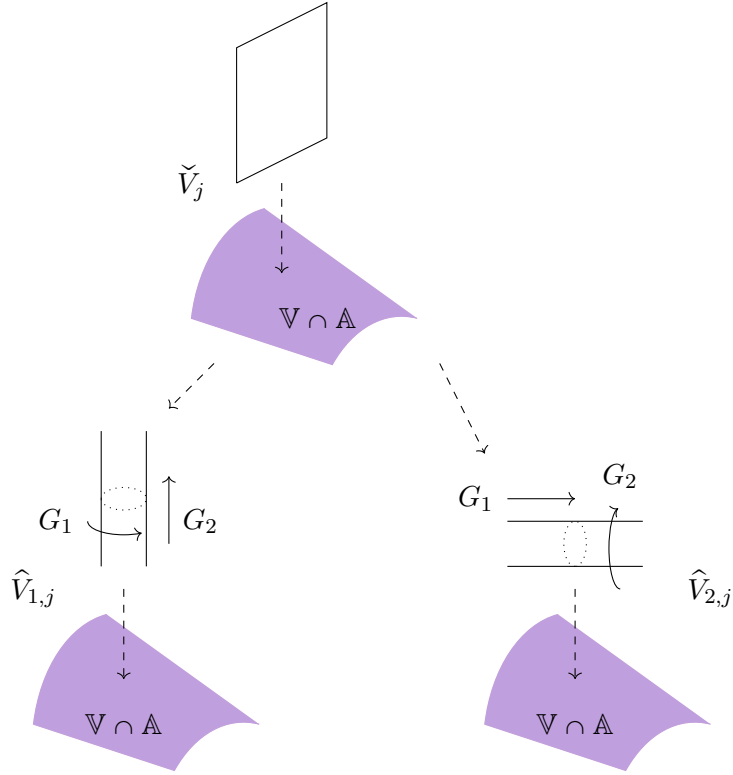
\begin{figure}[ht]
	\begin{tikzpicture}[scale=0.6]
		\filldraw[color=red!50,fill=red!50,opacity=0.5](0,2)
		arc [
        	start angle=100,
        	end angle=40,
        	x radius=-2cm,
        	y radius=3cm
    	]--(-5,2)
		arc [
        	start angle=170,
        	end angle=100,
        	x radius=2cm,
        	y radius=3cm
    	];
		\filldraw[color=blue!50,fill=blue!50,opacity=0.5](0,2)
		arc [
        	start angle=100,
        	end angle=40,
        	x radius=-2cm,
        	y radius=3cm
    	]--(-5,2)
		arc [
        	start angle=170,
        	end angle=100,
        	x radius=2cm,
        	y radius=3cm
    	];  

		\draw[dashed,->] (-3,5)--(-3,3);
		\draw (-3.5,8.5)--(-3.5,5.5);
		\draw (-2.5,8.5)--(-2.5,5.5);
		\draw[dotted] (-3.5,7) 
		arc [
        	start angle=0,
        	end angle=360,
        	x radius=-0.5cm,
        	y radius=0.25cm
    	];

		\draw[->](-3.8,6.5) 
		arc [
        	start angle=0,
        	end angle=-120,
        	x radius=-0.8cm,
        	y radius=0.3cm
    	];
		\node at (-4.5,6.5) {$G_1$};
		\draw[->] (-2,6)--(-2,7.5);
		\node at (-1.3,6.5) {$G_2$};

		\filldraw[color=red!50,fill=red!50,opacity=0.5](10,2)
		arc [
        	start angle=100,
        	end angle=40,
        	x radius=-2cm,
        	y radius=3cm
    	]--(5,2)
		arc [
        	start angle=170,
        	end angle=100,
        	x radius=2cm,
        	y radius=3cm
    	];
		\filldraw[color=blue!50,fill=blue!50,opacity=0.5](10,2)
		arc [
        	start angle=100,
        	end angle=40,
        	x radius=-2cm,
        	y radius=3cm
    	]--(5,2)
		arc [
        	start angle=170,
        	end angle=100,
        	x radius=2cm,
        	y radius=3cm
    	];  

		\draw[dashed,->] (7,5)--(7,3);
		\draw (5.5,5.5)--(8.5,5.5);
		\draw (5.5,6.5)--(8.5,6.5);
		\draw[dotted] (7,5.5) 
		arc [
        	start angle=90,
        	end angle=450,
        	x radius=0.25cm,
        	y radius=-0.5cm
    	];
		\draw[->](8,5) 
		arc [
        	start angle=90,
        	end angle=260,
        	x radius=0.25cm,
        	y radius=-0.9cm
    	];
		\node at (8,7.5) {$G_2$};
		\draw[->] (5.5,7)--(7,7);
		\node at (4.8,7) {$G_1$};

		\draw[dashed,->] (-1,10)--(-2,9);
		\draw[dashed,->] (4,10)--(5,8);

		\filldraw[color=red!50,fill=red!50,opacity=0.5](3.5,11)
		arc [
        	start angle=100,
        	end angle=40,
        	x radius=-2cm,
        	y radius=3cm
    	]--(-1.5,11)
		arc [
        	start angle=170,
        	end angle=100,
        	x radius=2cm,
        	y radius=3cm
    	];
		\filldraw[color=blue!50,fill=blue!50,opacity=0.5](3.5,11)
		arc [
        	start angle=100,
        	end angle=40,
        	x radius=-2cm,
        	y radius=3cm
    	]--(-1.5,11)
		arc [
        	start angle=170,
        	end angle=100,
        	x radius=2cm,
        	y radius=3cm
    	];  

		\draw[dashed,->] (0.5,14)--(0.5,12);
		\draw (-0.5,14)--(1.5,15)--(1.5,18)--(-0.5,17)--(-0.5,14);

		\node at (-3,2) {$\mathbb{V}\cap\mathbb{A}$};
		\node at (7,2) {$\mathbb{V}\cap\mathbb{A}$};
		\node at (1.3,11) {$\mathbb{V}\cap\mathbb{A}$};

		\node at (-5,5) {$\widehat{V}_{1,j}$};
		\node at (10,5) {$\widehat{V}_{2,j}$};
		\node at (-1.5,14) {$\widecheck{V}_{j}$};

	\end{tikzpicture}
	\caption{Intermediate covers $\widehat{V}_{1,j},\widehat{V}_{2,j}$ and the universal cover $\widecheck{V}_j$}
	 \label{Figure:Intermediate covers}
\end{figure}

Denote by $\widecheck \p_{ \alpha, j}$ the lift to $\widecheck V_j$ of $\widehat\p_{\alpha; \alpha, j}$.   Since $\widehat\p_{\alpha;\alpha, j}$ converges to $\widehat \p_{\alpha,\infty}$ on $\widehat V_{\alpha,\infty}$,  we see that  $\widecheck \p_{\alpha, j}$ on $\widecheck V_j$ also converges to the lift $\widecheck\p_{\alpha,\infty}$ of $\widehat \p_{\alpha,\infty}$ on $\widecheck V_\infty$. Obviously both $\widecheck\p_{1,\infty}$ and $\widecheck\p_{2,\infty}$ belong to the Lie algebra of $\widecheck H$. It follows that $\widecheck\p_{2,\infty}=\lambda\widecheck\p_{1,\infty}$ for some $\lambda\neq 0$. A crucial observation is that we must have $\lambda=\pm1$. Indeed, this follows from \eqref{e:lengthcurvature},  noticing that $|\kappa(x)|$ depends on the local orbit through $x$.

We remark that in the above we have passed to subsequences at various stages, but the derived properties hold on all the possible limits. 
By the continuity of angle between $\partial_1,\partial_2$,  we may without loss of generality assume that $\lambda=1$ on all possible limits. On any limit we also denote $\widecheck\p\equiv \widecheck\p_{1,\infty}=\widecheck\p_{2,\infty}$.  Then $\widecheck H$ is generated by $\widecheck\p$ and a translation vector field $\widecheck \xi$ with $\|\widecheck \xi\|=1$. Furthermore, the quotient $\widecheck V_\infty/\widecheck H$ is the subset $\mathbb{V}\cap\mathbb{A}$ of $\mathbb H$, and $\|\widecheck \p\|=\rho$.


\

In the rest of the proof we focus on the set $V=F^{-1}(\mathbb V)$. This is a (trivial) $T^2$ bundle over $\mathbb V$.  The function $\rho$ on $\mathbb H^\circ$ can be naturally viewed as a function on $V$.  The Riemannian metric $g$ gives rise to a family of flat metrics on $G=T^2$ parametrized by $\mathbb V$. We introduce the \emph{Gram matrix} $\mathbb G$ with $\mathbb G_{\alpha\beta}\equiv g(\p_\alpha, \p_\beta)$ for $1\leq\alpha, \beta\leq 2$.  Let $\tau$ be the function on $V$ such that ${\xi}\equiv\p_1-(1-\tau)\p_2$ is pointwise orthogonal to $\partial_{2}$, and let $\sigma\equiv \|\xi\|_{g}$. Notice that $\tau$ and $\sigma$ are both $T^2$ invariant so can be viewed as functions on $\mathbb V$.

As $j\rightarrow\infty$, passing to subsequences we know that  under the convergence of $\widecheck V_j$ to $\widecheck{V}_\infty$, the Killing fields  $\widecheck\p_{1, j}$ and $\widecheck\p_{2, j}$  both converge naturally to the same vector field $\widecheck\p$ on $\widecheck V_\infty$. On $\widecheck V_j$, the function $\rho$ is uniformly comparable to $2^j$. 
Denote  $\xi_j\equiv \p_{1, j}-(1-\tau)\p_{2, j}$ in the rescaled annulus $\widetilde{A}_j$.
\begin{lemma}Passing to subsequences and pulling back to $\widecheck{V}_j$, $2^j\sigma^{-1}\xi_j$ converges in $C^{2}$ to $\widecheck{\xi}$  as $j\rightarrow\infty$.
\end{lemma}
\begin{proof}
 To see this  we first choose a vector field $\xi_j'$ (for $j\gg1$) in the linear span of $\p_{1,j}$ and $\p_{2, j}$ on each $V_j$ such that after passing to subsequences and pulling back to $\widecheck{V}_j$ it converges in $C^{k_0}$ to the limit $\widecheck\xi$. Write $\xi_j'$ uniquely as $$\xi_j'=\sigma_j^{-1}(\partial_{1,j}-(1-\tau_{j})\partial_{2,j}).$$ Then we have 
$$\xi_j'=\sigma_j^{-1}(\xi_j+(\tau_j-\tau)\partial_{2,j}).$$
As $\xi_j'$ converges in $C^{k_0}$ to $\widecheck{\xi}$, we have 
\begin{align}g_j(\xi_j',\partial_{2,j})\to0, &&g_{j}(\xi_j',\xi_j')\to 1.	
\end{align}
 It follows that
\begin{align*}
	\sigma_j^{-1}{|\tau_j-\tau|}\to0,
	&& \sigma_j^{-1}\|\xi_j\|_{g_j}\to1,
\end{align*}
and $\sigma_j^{-1}\xi_j$ converges to $ \widecheck\xi$ in $C^{k_0}$. 
This implies that $\|{\xi}_j\|_{g_j}^{-1}{\xi}_j=2^j\sigma^{-1}\xi_j$ also converges in $C^2$ to $\widecheck{\xi}$.
\end{proof}
By construction since $\widecheck \xi$ is a translation vector field we must have
\begin{equation}\label{eqn1}
\lim_{\rho\rightarrow\infty}\tau=0.
\end{equation}
Denote $f\equiv\|\p_2\|_g$.  Then from the convergence of $\widecheck \p_{2,j}$ to $\widecheck\p$ one can see that 
\begin{equation}\label{eqn2}\lim_{\rho\rightarrow\infty}\rho^{-1}f=1.
\end{equation}
Observe that $2^{-2j}f\sigma$ is the area of the $T^2$ orbit in $V_j$. By assumption  we have 
\begin{equation} \label{eqn6}
	\lim_{\rho\rightarrow\infty}\rho^{-2}f\sigma=\lim_{\rho\rightarrow\infty}\rho^{-1}\sigma=0.
\end{equation}
Next we derive a few key limiting properties. 

\begin{proposition}\label{prop:codimension 2 collapse}
	We have 
	\begin{equation}\label{eqn3}
		\lim_{\rho\rightarrow\infty}\frac{\sigma}{\rho\tau}=0.
	\end{equation}
\end{proposition}

\begin{proof}
This follows from the fact that the collapsing is of codimension 2; in other words, the diameters of the $T^2$ orbits are of order $o(\rho)$. Scaling down the metric it suffices to show that a sequence of flat tori given by the quotient of standard $\R^2$ by the lattice $\Lambda$ generated by two vectors $(1, 0)$ and $(1-a, b)$ with $a, b\rightarrow 0$ has diameter goes to zero if and only if $b/a\rightarrow 0$. To see this fact, we simply notice that for $x_0\in [0, 1]$, the distance between $[(0, 0)]$ and $[(x_0, 0)]$ in $\R^2/\Lambda$ is given by the distance between $(0, 0)$ and the line $bx+ay=bx_0$ in $\R^2$. The latter goes to zero if and only if $b/a\rightarrow 0$. Write the $T^2$ orbits in $V_j$ as quotient of $\mathbb{R}^2$ by $\Lambda$ in terms of the basis ${\partial}_{2,j}$ and $2^j\sigma^{-1}{\xi}_j$ and apply the above. 
\end{proof}



\begin{proposition}
We have
    \begin{equation}\label{eqn4}
	    \lim_{\rho\rightarrow\infty}\frac{\rho\sigma_\rho}{\sigma}=0,
    \end{equation} 
	\begin{equation}\label{eqn5}
		\lim_{\rho\rightarrow\infty}\frac{\rho^2\tau_\rho}{\sigma}=0.
	\end{equation}
	\end{proposition}
	\begin{proof}
	Let $\p_\rho^{\sharp}$ be the vector field on $V$ given by the horizontal lift of the standard vector field $\p_\rho$ on $\mathbb{H}^\circ$. Now consider the convergence of $\widecheck V_j$ to $\widecheck V_\infty$. We see that the rescaled vector field $2^j\p_\rho^\sharp$ converges in $C^2$ to the radial vector field $\widecheck\p_\rho$  in the rotation plane of $\widecheck\p$ in $\widecheck V_\infty$. So the norm of the commutator $\|[2^{j}\p_\rho^\sharp, 2^j\sigma^{-1}\xi_j]\|_{g_j}$ converges to $\|[\widecheck\p_\rho, \widecheck\xi]\|=0$ as $j\rightarrow\infty$.
	Notice that since $\p_1$ and $\p_2$ are Killing fields and $[\p_1, \p_2]=\langle \p_\rho^\sharp,\p_1\rangle=\langle\p_\rho^\sharp, \p_2\rangle=0$, we have for $\alpha=1, 2$,
	$$\langle [ \p_\alpha,\p_\rho^\sharp], \p_\beta\rangle=\langle\mathcal L_{\p_\alpha}\p_\rho^\sharp, \p_\beta\rangle=\mathcal L_{\p_\alpha}\langle \p_\rho^\sharp, \p_\beta\rangle-(\mathcal L_{\p_\alpha}g)(\p_\rho^\sharp, \p_\beta)-\langle \p_\rho^\sharp, \mathcal L_{\p_\alpha}\p_\beta\rangle=0. $$
	Since $\xi$ is orthogonal to $\p_2$,
	$$\|[2^{j}\p_\rho^\sharp, 2^j\sigma^{-1}\xi]\|_{g_j}^2=2^{4j}\|(-\sigma_\rho\sigma^{-2}\xi+\sigma^{-1}\tau_\rho \p_2)\|_{g_j}^2= 2^{4j}((\frac{\sigma_\rho}{\sigma})^22^{-2j}+(\frac{\tau_\rho}{\sigma})^2 2^{-2j} f^2).$$
	Thus $\tau_\rho\sigma^{-1}=o(2^{-2j})$ and $\sigma_\rho\sigma^{-1}=o(2^{-j})$ on $V_j$. The conclusion then follows.
	\end{proof}

	Now using \eqref{eqn1}, \eqref{eqn6}, \eqref{eqn3}, \eqref{eqn4}, \eqref{eqn5} and the L'Hospital's rule we get 
\begin{align*}
	\infty=\lim_{\rho\to\infty}\frac{\tau}{\sigma/\rho}=\lim_{\rho\to\infty}\frac{\tau_\rho}{\sigma_\rho/\rho-\sigma/\rho^2}=\lim_{\rho\to\infty}\frac{\tau_\rho}{-\sigma/\rho^2}=0.
\end{align*}
This is a contradiction hence completes the proof of Theorem \ref{thm}.

	\begin{remark}
\eqref{eqn4} has a geometric explanation if we  consider the convergence to the asymptotic cone $\mathbb H$ as metric measure spaces.  On the one hand, we know from the local description of the universal covering geometry that the renormalized limit measure on $\mathbb H$ has  a density of the form $d\nu= \rho d\rho dz$. On the other hand, by Fukaya's theorem \cite{Fukaya} we know that up to a multiplicative constant $f\sigma$ converges to $\rho$.  Then \eqref{eqn4} can also be obtained by comparing these two facts. 
\end{remark}

\section{Further discussion}
	In this section we construct an explicit sequence of Riemannian metrics $g_i$ on $\mathfrak A=S^3\times [2^{-1}, 2]$ with $\sup_{\mathfrak A}|Rm_{g_i}|\rightarrow 0$ which collapse to the annulus  $\mathbb A\equiv\{2^{-1}\leq r\leq 2\}$ in $\mathbb H$. This shows that Theorem \ref{thm} is of a global nature. We will also explain the relevance to the asymptotic curvature gap conjecture of Petrunin-Tuschmann \cite{PT}.

\subsection{A model flat metric}\label{subsec:model flat metric}
Denote the flat product metric $\C^2=\C\times \C$. Given $\alpha, \sigma\in [0,1)$, we consider the Euclidean motion $\gamma: (z_1, z_2)\mapsto (z_1+\sigma, e^{2\pi \sqrt{-1}\alpha}z_2)$. Let $\mathcal X_{\alpha, \sigma}$ be the quotient of $\C^2$ by the $\mathbb Z$-action generated by $\gamma$. Then $\mathcal X_{\alpha, \sigma}$ is a complete flat manifold with cubic volume growth, but the asymptotic geometry depends on the rationality of $\alpha$.  When $\alpha$ is rational  we write $\alpha=p/q$ for $p, q$ coprime, then $\mathcal X_{\alpha, \sigma}$ has asymptotic cone  given by  the product $\C_\alpha\times \R$, where $\C_\alpha$ is a 2 dimensional flat cone with angle $2\pi/q$.  
When $\alpha$ is irrational the end geometry is quite different. This is similar to the example of 3 dimensional flat manifolds studied by Gromov (see \cite{Minerbe, Chen-Li}). One can check that $\mathcal X_{\alpha, \sigma}$ is still conical but the asymptotic cone is instead the half plane $\mathbb H$.

Notice that $\mathcal X_{\alpha, \sigma}$ admits a $T^2$ symmetry, with the first $S^1$ action inherited from the rotation on the second $\C$ factor, and the second $S^1$ action given by $e^{2\pi\sqrt{-1}t}\cdot (z_1, z_2)=(z_1+t\sigma, e^{2\pi\sqrt{-1}\alpha t}z_2)$.  So we may rewrite the flat metric on $\mathcal X_\alpha$ in terms of the Gram matrix $\mathbb G$ as $$g=d\rho^2+dz^2+\sum_{\alpha, \beta} \mathbb G_{\alpha\beta}d\phi_\alpha\otimes d\phi_\beta,$$ where $(\rho, z)$ denote the standard coordinates on $\mathbb H$.  One can compute that 
\begin{equation}
	\mathbb G=\begin{pmatrix}
	\rho^2 & \alpha \rho^2\\
	\alpha\rho^2 & \alpha^2\rho^2+\sigma^2
\end{pmatrix}.
\end{equation}
Notice $\det(\rho^{-1}\sigma^{-1}\mathbb G)=1$. It is a general fact that 4 dimensional toric Ricci-flat metrics give rise to integrable systems (see for example \cite{Kunduri}). One can check that in our setting the normalized Gram matrix $\sigma^{-1}\rho^{-1}\mathbb G$ defines an axi-symmetric harmonic map from $\mathbb R^3$ (obtained by rotating $\mathbb H$ around the $z$-axis) into the symmetric space $SL(2; \mathbb R)/SO(2)$; the latter can be identified with the hyperbolic plane $\mathcal H=\{(X, Y)|X>0\}$ via
$$\left(\begin{matrix}
X+X^{-1}Y^2 & X^{-1}Y\\ X^{-1}Y & X^{-1}
\end{matrix}\right).$$

Following usual terminologies in the literature we call $\mathcal X_{\alpha, \sigma}$  an \emph{AF} model end. An AF manifold is by definition a complete Riemannian manifold that is asymptotic to an AF model end at a polynomial rate.  For example, the Schwarzschild gravitational instanton and the Kerr gravitational instantons are AF. More delicate examples of AF gravitational instantons were constructed explicitly by Chen-Teo \cite{Chen-Teo} using the technique of inverse scattering transform. 

 By definition an AF manifold is conical $\mathcal{AF}$.  Notice that an AF end  is not simply-connected, so we do not arrive at a contradiction with Theorem \ref{thm}. 
  The manifolds $\mathcal X_{\alpha, \sigma}$ for $\alpha$ irrational serve as local models for the construction in the next subsection.

\subsection{Simply-connected domains collapsing to \texorpdfstring{$\mathbb{A}$}{}}
\label{subsec:collapse to annulus}
We will adopt the notation in Section \ref{sec:proof} and consider only $T^2$ invariant metrics. The question reduces to constructing suitable family of flat metrics on $T^2$ with prescribed behavior. We introduce polar coordinates $(r, \theta)$ on $\mathbb H$ so that $\rho=r\cos\theta$, $z=r\sin\theta$ with $\theta\in[0, \pi]$.  

Around the boundary when $\theta$ is close to $0$ and $\pi$ we can use  the annuli in the model flat end $\mathcal X_{\alpha, \sigma}$. In order to create collapsing metrics on $\mathfrak A$, we need to glue together annuli in two different such ends. 

We introduce the model normalized Gram matrices 
\begin{equation}
	\mathbb{G}^1_{\sigma,\tau}\equiv\left(\begin{matrix}
	\frac{1}{\sigma}r\sin\theta &\frac{1-\tau}{\sigma}r\sin\theta\\
    \frac{1-\tau}{\sigma}r\sin\theta& \frac{(1-\tau)^2r^2\sin^2\theta+\sigma^2}{\sigma r\sin\theta}
	\end{matrix}\right),
\end{equation}
\begin{equation}
	\mathbb{G}^2_{\sigma,\tau}\equiv\left(\begin{matrix}
		\frac{(1-\tau)^2 r^2\sin^2\theta+\sigma^2}{\sigma r\sin\theta} &\frac{1-\tau}{\sigma} r\sin\theta\\
		\frac{1-\tau}{\sigma} r\sin\theta& \frac{1}{\sigma} r\sin\theta
	\end{matrix}\right).
\end{equation}
Notice that the metric $d r^2+ r^2d\theta^2+\sigma r\sin\theta \mathbb{G}^\alpha_{\sigma,\tau}$ are both isometric to $\mathcal X_{1-\tau, \sigma}$. The difference between the two is that for each $\alpha$, the rotational Killing field for $\mathbb G^\alpha_{\sigma, \tau}$ is given by $\p_{\phi_\alpha}$. 
The idea is to patch these two models together to construct a global normalized Gram matrix $\mathbb{G}$, so that $\mathbb{G}=\mathbb{G}^1_{\sigma,\tau}$ over $\theta\in[0,\frac{\pi}{3}]$ and $\mathbb{G}=\mathbb{G}^2_{\sigma,\tau'}$ over $\theta\in[\frac{2\pi}{3},\pi]$. Then by construction the topology will be simply-connected. 

The key point is  on the choice of the appropriate parameters  $\sigma,\tau,\tau'$.
Viewing  $\mathbb{G}^\alpha_{\sigma, \tau}$ as maps into $\mathcal H$, for $ r\in[2^{-1},2]$ we have 
\begin{align}
	\mathbb{G}^1_{\sigma,\tau}( r)&:[0,\pi]\mapsto (X,Y)=\left(\frac{\sigma r\sin\theta}{(1-\tau)^2 r^2\sin^2\theta+\sigma^2},\frac{(1-\tau) r^2\sin^2\theta}{(1-\tau)^2 r^2\sin^2\theta+\sigma^2}\right),\\
	\mathbb{G}^2_{\sigma,\tau}( r)&:[0,\pi]\mapsto (X,Y)=\left(\frac{\sigma}{ r\sin\theta},1-\tau\right).
\end{align}
The images are contained in geodesics in $\mathcal H$. Indeed, the image of $\mathbb{G}^1_{\sigma,\tau}( r)$ is part of the half circle from $(0,0)$ to $(0,\frac{1}{1-\tau})$ and the image of $\mathbb{G}^2_{\sigma,\tau}( r)$ is part of the straight line $Y=1-\tau$ (see Figure \ref{Figure:hyperbolic space}).

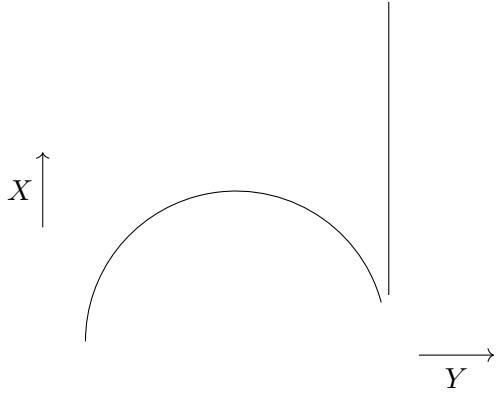
\begin{figure}
\begin{tikzpicture}
	\draw (-2,0) (2,0) arc(15:180:2);
	\draw (2.1,0.1)--(2.1,4);

	\draw[->] (2.5,-0.7)--(3.5,-0.7);
	\node at (3,-1) {$Y$};

	\draw[->] (-2.5,1)--(-2.5,2);
	\node at (-2.8,1.5) {$X$};
\end{tikzpicture}
\caption{Images of $\mathbb{G}_{\sigma,\tau}^\alpha$}
\label{Figure:hyperbolic space}
\end{figure}

Now for $\tau\ll1$ we set $1-\tau'=\frac{1}{1-\tau}$ in order to make the half circle and the line tangential. Then we choose $\sigma\ll\tau\rightarrow 0$ so that each $T^2$ collapses to a point. We  patch $\mathbb{G}^1_{\sigma,\tau}( r)$ and $\mathbb{G}^2_{\sigma,\tau'}( r)$ by interpolating them over $\theta\in[\frac{\pi}{3},\frac{2\pi}{3}]$. 
Consider the isometry of $\mathcal H$ given by 
$$\Phi:(X,Y)\mapsto \left(\frac{1}{\sigma}X,\frac{1}{\sigma}(Y-\frac{1}{1-\tau})\right).$$ 
 Given any normalized Gram matrix $\mathbb{G}$ associated to $(X, Y)\in \mathcal H$,  we denote by $\Phi \mathbb G$ the normalized Gram matrix corresponding to $\Phi(X, Y)\in \mathcal H$. 
This is nothing but rewriting $\mathbb{G}$ under the basis $\frac{1}{\sqrt\sigma}(\partial_{1}-\frac{1}{1-\tau}\partial_{2}),\sqrt\sigma\partial_{2}$. We have \begin{align*}
	\Phi\mathbb G_{\sigma,\tau}^1( r)&=\left(\frac{ r\sin\theta}{(1-\tau)^2 r^2\sin^2\theta+\sigma^2},-\frac{\sigma}{(1-\tau)^2 r^2\sin^2\theta+\sigma^2}\right),\\
	\Phi \mathbb G_{\sigma,\tau'}^2( r)&=\left(\frac{1}{ r\sin\theta},0\right).
\end{align*}

Now by a straightforward interpolation one can find a smooth family of curves $\mathbb G'_{\sigma, \tau}: \{\frac{1}{2}\leq r \leq 2]\}\subset \mathbb H\rightarrow \mathcal H$ such that $\mathbb G_{\sigma, \tau}'=\Phi\mathbb G^1_{\sigma, \tau}$ for $\theta\in [0, \frac{\pi}{3}]$, $\mathbb G'_{\sigma, \tau}=\Phi \mathbb G^2_{\sigma, \tau'}$ for $\theta\in [\frac{2\pi}{3}, \pi]$, and for $\theta\in [\frac{\pi}{3}, \frac{2\pi}{3}]$ we have $\mathbb G'_{\sigma, \tau}$ converge smoothly to $(\frac{1}{r\sin\theta}, 0)$ as $\sigma\ll\tau\rightarrow0$. 
%
%
%
Then we define $\mathbb G_{\sigma, \tau}\equiv \Phi^{-1}\mathbb G'_{\sigma, \tau}$. It induces the metric on $S^3\times[2^{-1},2]$ via
\begin{equation}
	g_{\sigma,\tau}=d r^2+ r^2d\theta^2+\sigma r\sin\theta \mathbb{G}_{\sigma,\tau}.
\end{equation}

We claim that as $\sigma\ll\tau\to0$, $g_{\sigma,\tau}$  is collapsing to $\mathbb A$ with $\sup_{\mathfrak A}|Rm_{g_{\sigma, \tau}}|\rightarrow 0$. For this we study the limit geometry of local universal covers.  It is clear that for $\theta\in[0,\frac{\pi}{3}]$ and $[\frac{2\pi}{3},\pi]$ the limit geometry is given by the model flat ends. We only need to consider the region when $\theta\in[\frac{\pi}{3},\frac{2\pi}{3}]$. 

Under the basis $\{\frac{1}{\sqrt\sigma}(\partial_1-\frac{1}{1-\tau}\partial_2),\sqrt{\sigma}\partial_2\}$ we have
\begin{equation}\label{eq:gram matrix moved into hyperbolic space}
	\sigma r\sin\theta\Phi \mathbb G_{\sigma,\tau}=\sigma r\sin\theta\left(\begin{matrix}
		X'+X'^{-1}Y'^2 & X'^{-1}Y'\\ X'^{-1}Y' & X'^{-1}
	\end{matrix}\right),
\end{equation}
where $(X',Y')$ converge smoothly to $(\frac{1}{ r\sin\theta},0)$. 
This precisely means that under $\{\frac{1}{\sigma}(\partial_{1}-\frac{1}{1-\tau}\partial_2),\partial_{2}\}$, as $\sigma\ll\tau\to0$, the metrics $g_{\sigma,\tau}$ when pulled back to the universal covers converge smoothly to the standard flat metric in $\mathbb R^3\times \R$:  
$$d r^2+ r^2d\theta^2+\left(\begin{matrix} 1&0\\0& r^2\sin^2\theta
\end{matrix}\right).$$
The  condition that  $\sigma\ll\tau$ ensures that each $T^2$ collapses to a point as in Proposition \ref{prop:codimension 2 collapse}, so the limit is $\mathbb{A}$.

\subsection{No gap for asymptotic curvature}

It is tempting to apply the construction in the previous subsection to form a counterexample to Conjecture \ref{conj:PT}. However, as Theorem \ref{thm} shows we know this is impossible -- the issue arises when we take the limit as $\rho\rightarrow\infty$ since the asymptotic properties deduced in Section 2 can not be simultaneously achieved (due to the L'Hospital rule!). Nevertheless, if we allow some flexibility then one can indeed construct examples satisfying most of the other asymptotic properties. We may relax the asymptotical flatness property as follows:

\begin{proposition}
	For any $\epsilon>0$ there is a complete Riemannian metrics $g_\epsilon$ on the standard $\mathbb R^4$ with quadratic curvature decay and $A(M, g_\epsilon)\leq \epsilon$ which is asymptotic to the cone $\mathbb H$. 
\end{proposition}
\begin{remark}
	Since we can not take $\epsilon=0$,  this shows there is no gap property for asymptotic curvature if we fix the asymptotic cone. i.e., $g_\epsilon$ can not be perturbed to an $\mathcal {AF}$ metric with the same asymptotic cone. In particular, it gives a negative answer to a version of Question 2 in \cite{PT}. 
\end{remark}
\begin{proof}
Note the for the construction in Section \ref{subsec:collapse to annulus} we have  chosen  $\sigma$ and $\tau$ to be constants. Now we may choose them to depend on $r$. We will still require $\sigma/\tau\to0$ as $ r\to\infty$ so that the asymptotic cone is given by collapsing the $T^2$ orbits. We then obtain the normalized Gram matrix $\mathbb{G}_{\sigma( r),\tau( r)}$, and define a metric 
$$g=d r^2+ r^2d\theta^2+\sigma r\sin\theta\mathbb{G}_{\sigma,\tau}$$
when $ r>1$.  Set $\sigma=\epsilon\frac{ r}{\log r}$ for $\epsilon>0$ small and $\tau=\frac{1}{(\log r)^{1/2}}$.
From the construction in Section \ref{subsec:collapse to annulus}, we see that over $ r\in(2^{j-1},2^{j+1})$, the frame $\{2^{j}\partial_ r,\partial_\theta,2^j\frac{1}{\sigma}(\partial_1-\frac{1}{1-\tau}\partial_2),\partial_2\}$ along with the rescaled metric $2^{-j}g$ converges smoothly to a frame $\{\partial_ r,\partial_\theta,\xi_\infty,\partial_\infty\}$, under which the metric converges to 
$$\left(\begin{matrix}
	1 & & &\\
	 & r^2 & &\\
	& & 1 &\\
	& & &  r^2\sin^2\theta
\end{matrix}\right).$$
One can check that the Lie brackets between the limit frames are all zero except\begin{align*}
[\partial_ r,\xi_\infty]=\frac{1}{ r}\epsilon\xi_\infty.
\end{align*}
In particular, the asymptotic curvature can be made arbitrarily small by taking $\epsilon$ small.  Filling in suitably we get smooth Riemannian metrics on $\mathbb{R}^4$.
\end{proof}

In the above construction we may also choose functions $\sigma(r)$ and $\tau(r)$ to converge to zero in a comparable rate as $r\rightarrow\infty$. In this way we can achieve asymptotical flatness, but the asymptotic cone will be different. This is indeed the case for some explicit asymptotically flat metrics on $\mathbb R^4$, for example the Taub-NUT gravitational instantons. The asymptotic cone is given by $\mathbb R^3$ and the collapsing to the asymptotic cone is only of codimension 1.  

These examples demonstrate the subtlety in studying Conjecture \ref{conj:PT}.

\section{Gravitational instantons}
\label{sec:discussion}

Recall that in this paper by a gravitational instanton we mean a complete noncompact Riemannian 4-manifold $(M, g)$ with vanishing Ricci curvature and with finite energy (i.e. $\int_M |Rm_g|^2<\infty$).  For simplicity of discussion we will always assume $M$ is oriented. Our original motivation to study Conjecture \ref{conj:PT} arises from the aim of understanding the asymptotic geometry of  gravitational instantons. The topic of gravitational instantons has been extensively studied for a long time in both the mathematical and physics literature.  In this section we include a brief summary of what is known to date about gravitational instantons and list some open questions.

Recall that we may view the Riemann curvature tensor of a general Riemannian manifold as a self-adjoint endomorphism $\mathcal R$ on the bundle $\Lambda^2$ of 2-forms. Standard representation theory of $SO(n)$ decomposes $\mathcal R$ into the sum of 3 components: one involving the scalar curvature $S$, one involving the trace-free Ricci curvature $\overset{\circ}{Ric}$ and one involving the Weyl curvature $W$. In dimension 4 since $SO(4)$ is a double cover of the product $SO(3)\times SO(3)$, we have a refined decomposition.  The bundle $\Lambda^2$  splits as the direct sum of the bundle of self-dual and anti-self-dual 2-forms $\Lambda^2=\Lambda^+\oplus \Lambda^-$. Accordingly we may write $\mathcal R$ as 
\begin{equation*}
\mathcal R=\begin{pmatrix}
  W^++\frac{S}{12} & \overset{\circ}{Ric}\\ 
  \overset{\circ}{Ric} & W^-+\frac{S}{12}
\end{pmatrix}.
\end{equation*}
  At each point we may view $W^+$ as a trace-free symmetric  $3\times 3$ matrix. Following \cite{Li1} we distinguish gravitational instantons into 3 types. 

\

\begin{itemize}
	\item Type I: $W^+$ vanishes identically.  This condition is equivalent to that $g$ being locally \emph{hyper\"ahler}. Namely, locally there are 3 compatible complex structures $J_1, J_2, J_3$ satisfying the quaternion relations $J_1J_2=-J_2J_1=J_3$ which are all parallel with respect to the Levi-Civita connection and which define the given orientation. 
	
	\
	
	\item Type II: $W^+$  has exactly 2 distinct eigenvalues everywhere. It was first discovered by Derdzi\'nski \cite{derdzinski} that this condition is equivalent to saying locally there is a complex structure $J$ such that $g$ is hermitian under it. Passing to a double cover if necessary, there exists a global complex structure $J$.  Furthermore, on the double cover, the conformal metric $\tilde g\equiv (24|W^+|^2)^{\frac{1}{3}}g$ is K\"ahler extremal in the sense of Calabi and has vanishing Bach tensor. The extremal vector field for $\tilde g$ is a non-zero Killing field for $g$. 
	
	\
	
	\item Type III: $W^+$  has generically 3 distinct eigenvalues. 
\end{itemize}
\

Notice that the Type of a gravitational instanton depends on the choice of orientation.  For example, 
	the Eguchi-Hanson metric and the Taub-NUT metric are Type I with respect to the standard orientation but are Type II with respect to the reversed orientation.

By definition (global) hyperk\"ahler gravitational instantons are of Type I. These are among the most well-studied classes. In fact, often in the literature the hyperk\"ahler condition is taken for granted in the definition of gravitational instantons. There are many different constructions of hyperk\"ahler gravitational instantons and they exhibit a variety of interesting behavior at infinity.

By Bando-Kasue-Nakajima \cite{BKN} a gravitational instanton with Euclidean volume growth is ALE, i.e., it is asymptotic to a flat cone $\R^4/\Gamma$ for some $\Gamma\in SO(4)$ at a polynomial rate. ALE gravitational instantons have finite fundamental groups. All the known examples of ALE gravitational instantons are finite quotients of hyperk\"ahler gravitational instantons; indeed they have been classified \cite{Kronheimer, Souvaina, Wright}. The following is a longstanding question

\begin{conjecture}[Bando-Kasue-Nakajima \cite{BKN}]\label{BKN}
	All ALE gravitational instantons are of Type I up to reversing the orientation. 
\end{conjecture}

If there is a hypothetical Type II ALE gravitational instanton, then its structure group $\Gamma$ at infinity is conjugate to a subgroup of $U(2)$. In \cite{Li1} it is proved that 

\begin{theorem}[\cite{Li1}]\label{thn: Li1}
	The only Type II ALE gravitational instanton with structure group $\Gamma\subset SU(2)$ is given by the Eguchi-Hanson metric with the reversed orientation.
\end{theorem}

This gives supporting evidence to Conjecture \ref{BKN}. It is likely that the result holds without the restriction  $\Gamma\subset SU(2)$. If so then Conjecture \ref{BKN} amounts to ruling out gravitational instantons which are Type III for both orientations.

\

When a gravitational instanton does not have Euclidean volume growth, the geometry at infinity is much richer. In the hyperk\"ahler case we have

\begin{theorem}[\cite{SZ}] \label{SZ}
	A non-flat hyperk\"ahler gravitational instanton must be conical and asymptotic to one of the following 6 families of model ends:  $ALE, ALF, ALG, ALH, ALG^*, ALH^*$. 
\end{theorem}

Roughly speaking, the $AL\sharp(^*)$ model ends have volume growth of order $(4-(\sharp-E))$. The 4 families without $^*$ are  locally given by flat torus bundles over a domain in the Euclidean space; the 2 families with $*$ have inhomogeneous geometries at infinity.  A key common feature is that all the model ends admit \emph{local nilpotent symmetries}.

The proof of Theorem \ref{SZ} exploits a combination of Riemannian collapsing theory of Cheeger-Fukaya-Gromov and the technique of  adiabatic PDE perturbation. It came as a by-product of a more general study of collapsing geometry of 4 dimensional hyperk\"ahler metrics.   One simple but central fact is that a hyperk\"ahler metric with a continuous symmetry is locally given by a positive harmonic function on $\R^3$. There is an earlier work Chen-Chen \cite{Chen-Chen1} proving a special case of theorem under the extra  assumption of faster than quadratic curvature decay. Notice that $ALE, ALF, ALG, ALH$ have faster than quadratic curvature decay,  $ALG^*$ is $\mathcal{AF}$ but not strongly $\mathcal{AF}$, and $ALH^*$ is not even $\mathcal{AF}$.

Given Theorem \ref{SZ} the next question is to classify all hyperk\"ahler gravitational instantons asymptotic to a given model end. This is essentially completed due to the work of many people, see for example \cite{Kronheimer, Chen-Chen2, Chen-Chen3, SZ19, SZ, HSVZ2, CJL1, LL, CVZ}. In particular, Torelli type theorems are proved, which gives a classification of AL$\sharp$ hyperk\"ahler gravitation instantons in terms of the topological data of periods. The strategy is to use the prescribed asymptotics to compactify the underlying complex manifold (for particular choices of the complex structure) into a K\"ahler surface by adding certain divisor at infinity, and then describe the metric in terms of the (suitably generalized) \emph{Tian-Yau} construction. One idea to prove the Torelli theorem is to use gluing method to prove that an AL$\sharp$ gravitational instanton (satisfying certain topological bound when $\sharp=E, F$) arises as models of singularity formation for a sequence of hyperk\"ahler metrics on the K3 manifold, and then refer to the classical Torelli theorem for K3 manifolds. The original proof of the surjectivity part for K3 Torelli makes use of Yau's proof of the Calabi conjecture; more recently Liu \cite{Liu} gave a new proof using instead the compactness result of \cite{SZ}.

 The upshot is that we have a relatively complete understanding of hyperk\"ahler gravitational instantons. The next natural question is
 \begin{problem}
 	Classify Type I gravitational instantons. 
 \end{problem}
  Notice that the universal cover of a Type I gravitational instanton is always a complete hyperk\"ahler manifold, but the energy is finite only when the fundamental group is finite. 
  In the case of finite fundamental group the problem reduces to the more manageable 
 
 \begin{question}
 	Classify free action of finite groups on hyperk\"ahler gravitational instantons. 
 \end{question}

The case of infinite fundamental group is more subtle. One example is given by the flat manifold $Y_{\alpha, \sigma}$ introduced in Section 3.1. The asymptotic geometry at infinity depends on the rationality of $\alpha$. 
\begin{question}
Is there a non-flat Type I gravitational instanton $(X,g )$  with $\pi_1(X)$ infinite? 	
\end{question}

If such an example exists, then the universal cover $\tilde X$ must have infinite energy. The only known examples of infinite energy hyperk\"ahler 4-manifold are constructed using the Gibbons-Hawking ansatz \cite{AKL, Hattori}. 

In general the study of infinite energy complete Ricci-flat metrics is an interesting problem. We have the following conjecture in \cite{SZ}

\begin{conjecture}[Finite energy vs finite topology]
A complete Ricci-flat 4-manifold has finite energy if and only if it has finite topological type.	
\end{conjecture}

The study of Type II gravitational instantons has only emerged in the recent few years. The known examples are all  ALF or AF, except the Eguchi-Hanson metric with the reversed orientation, which is ALE. Recall that AF model ends are introduced in Section \ref{subsec:model flat metric}.
Known ALF and AF Type II gravitational instantons include 
\begin{itemize}
\item  the Taub-NUT metric on $\R^4$ (with the reversed orientation);
\item  the Taub-Bolt metric on $\C\mathbb P^2\setminus\{point\}$ (with both orientations);
\item the Kerr metrics on $\R^2\times S^2$ (including the Schwarzschild metrics);
\item the Chen-Teo metrics $\C\mathbb P^2\setminus S^1$ (\cite{Chen-Teo, AA}).	
 \end{itemize}
 The first two classes are ALF and the second two classes are AF. 
 These metrics can all be written down explicitly in terms of algebraic formulae and they are all toric (i.e. their isometry group contains a 2 dimensional torus $T^2$). It was discovered by Biquard-Gauduchon \cite{BG} that all of them can be re-constructed using the LeBrun-Tod ansatz in complex geometry in terms of axi-symmetric harmonic functions in $\R^3$, and moreover any toric Type II ALF/AF gravitational instanton must belong to one of the 4 above families.
 
If $(M, g)$ is Type II, then suppose there is a global complex structure $J$, which can always be achieved by passing to a double cover. We denote $\lambda\equiv \sqrt{24}|W^+(g)|>0$, and $\tilde g=\lambda^{2/3} g$ the conformal metric.  Then $\tilde g$ is K\"ahler with respect to an integrable complex structure $J$. In fact, $\tilde g$ is extremal in the sense of Calabi, namely, the extremal vector field $\mathcal K\equiv J\nabla_{\tilde g}S(\tilde g)$ is a non-zero holomorphic Killing field. Moreover, we have $S(\tilde g)^2=\lambda^{2/3}$.  The following is an easy fact. 

\begin{lemma}\label{lem:scalar curvature positive}
	$S(\tilde g)>0$. 
\end{lemma}
\begin{proof}

	Suppose otherwise, then $S(\tilde g)<0$.  By Cheeger-Tian's $\epsilon$-regularity theorem \cite{CT} we know $|Rm(g)|$ decays quadratically fast at infinity. In particular we have $S(\tilde g)\rightarrow 0$. So $S(\tilde g)$ must achieve a global negative minimum at some point, say $p$, which is the global maximum of $S(\widetilde{g})^{-1}$. 	From the formula of scalar curvature under conformal change we have 
	$$S(\widetilde{g})^3\left(-6\Delta_{\widetilde{g}}+S(\widetilde{g})\right)S(\widetilde{g})^{-1}=0.$$
	But at $p$ this gives $\Delta_{\tilde g}(S(\tilde g)^{-1})\geq 0$. Contradiction. 
\end{proof}

Similar to classification of asymptotic geometry of hyperk\"ahler metrics at infinity, Theorem \ref{SZ}, we have 

\begin{theorem}[\cite{Li2}]
	A Type II Ricci-flat metric with $\int|Rm|^2<\infty$ on an end that is complete at infinity must be conical and asymptotic to one of the following families of model ends: ALE, ALF, AF, skewed special Kasner, $\text{ALH}^{*}$, or further $\mathbb{Z}_2$ quotients of the AF, skewed special Kasner, $\text{ALH}^{*}$ models.
\end{theorem}

Notice that the $\text{ALH}^*$ model ends are Type II under the orientation opposite to their hyperk\"ahler orientation. The \textit{skewed special Kasner models} are quotients of the following special Kasner metric
\begin{equation}
	d\rho^2+\rho^{4/3}dx_1^2+\rho^{4/3}dx_2^2+\rho^{-2/3}dx_3^2
\end{equation}
on $[1,\infty)\times T^3$. These model ends are not $\mathcal{AF}$. The special Kasner metric is Type II under both orientations, since it is hermitian under the complex structures $J_1:d\rho\mapsto\rho^{-1/3}dx_3,dx_1\mapsto dx_2$ and $J_2:d\rho\mapsto-\rho^{-1/3}dx_3,dx_1\mapsto dx_2$. Note that there is indeed a family of Riemannian Kasner metrics that are Ricci-flat, but only the specific one mentioned above is Type II. That is why we refer to it as \textit{special Kasner}.
Notice that by Lemma \ref{lem:scalar curvature positive} the skewed special Kasner models and $\text{ALH}^*$ models cannot be filled in as Type II gravitational instantons.  We moreover have

\begin{theorem}[\cite{Li2}]\label{thm:Li2 collapsed}
	A Type II gravitational instanton with non-Euclidean volume growth must belong to the above 4 families. 
\end{theorem}

This result particularly confirms a conjecture of Aksteiner-Andersson \cite{AA}. The proof depends on studying the compactification of the underlying complex manifold to a log del Pezzo surface, the theory of complex surfaces, and the geometry of complete  extremal K\"ahler metrics. Notice that the conformal K\"ahler metric associated to the model ALF/AF end is a Bach-flat extremal K\"ahler metric with a Poincar\'e type cusp. 

In general we expect

\begin{conjecture}[Type II classification]
	All Type II gravitational instantons are toric and must be ALF or AF, except the Eguchi-Hanson metric with the reversed orientation.
\end{conjecture}

By Theorem \ref{thm:Li2 collapsed}, it suffices to study Type II ALE gravitational instantons. Note that the partial conclusion of \cite{Li1}, Theorem \ref{thn: Li1}, is that for a Type II ALE gravitational instanton, when the structure group at infinity is in $SU(2)$, it must be the reversed Eguchi-Hanson.

The remaining case of  Type III gravitational instantons are poorly understood at present. Imposing toric symmetries there are relations to integrable systems and axi-symmetric harmonic maps, but there is no general theory yet. The only example to our knowledge is constructed by Khuri-Reiris-Weinstein-Yamada \cite{KRWY}, which is asymptotic to the above special Kasner end (which is Type II). It has quadratic volume growth and admits an effective $T^2$ action. But these examples are not simply-connected; indeed they have fundamental group $\mathbb Z$ with universal cover unwrapping one of the circle direction. They are asymptotic to a ray but they are not $\mathcal{AF}$. 

\begin{question}
Are there simply-connected Type III gravitational instantons? 	
\end{question}
\begin{question}
Are there  Type III $\mathcal{AF}$ gravitational instantons? 	
\end{question}

As a first step towards the classification of Type III gravitational instantons, it is important to understand their asymptotic behavior. 
Under the assumption of faster than quadratic curvature decay and a technical assumption regarding asymptotic holonomy, Chen-Li \cite{Chen-Li} obtained an analogous result to Chen-Chen \cite{Chen-Chen1}. These ends are given by flat torus bundles over flat domains. There are still some fundamental questions that remain to be answered.

 \begin{question}[Conicity]
	Is every gravitational instanton  conical?  
\end{question}
Under the assumption of being strongly $\mathcal {AF}$ this is known by the general result of Kasue \cite{Kasue}. Even under the assumption of $\mathcal{AF}$ this question is open in general. 

To connect with topology, recall that even the standard smooth $\R^4$ admits more than one gravitational instantons, namely, the standard flat metric and the Taub-NUT metric.

\begin{question}
Classify gravitational instantons on topological $\R^4$ and other manifolds with small topology. 	
\end{question}

As a consequence of Theorem \ref{thm} and the result in \cite{Chen-Li} we know 

\begin{corollary}
	The only gravitational instantons on a topological $\R^4$ with curvature faster than quadratic decay are given by the flat metric and the Taub-NUT metric up to scaling.  
\end{corollary}

\begin{remark}
Without the finite energy condition, the question is more complicated. We do not know whether an exotic $\R^4$ admit an complete metric with vanishing (or non-negative) Ricci curvature. 
\end{remark}

\bibliographystyle{plain}
\bibliography{reference}

\end{document}